\documentclass{article}
\usepackage{amsthm}
\usepackage{amsfonts}
\usepackage{amssymb}
\usepackage{mathrsfs}
\usepackage{amsmath}
\usepackage{graphicx} 
\usepackage{hyperref}
\usepackage{xcolor}
\hypersetup{
	colorlinks,
	linkcolor={black},
	citecolor={black},
	urlcolor={black}
}
\usepackage{cleveref}

\usepackage{verbatim}
\usepackage{pifont}
\usepackage{prettyref}
\usepackage{float}
\usepackage{mathrsfs}
\usepackage{enumitem}
\usepackage{algorithm2e}
\usepackage{amstext}
\usepackage{amsthm}
\usepackage{amssymb}
\usepackage{stmaryrd}
\usepackage{microtype}
\usepackage{hyperref}
\usepackage{tikz}
\usetikzlibrary{graphs}
\newtheorem{theorem}{Theorem}[section]
\newtheorem{lemma}[theorem]{Lemma}

\newtheorem{proposition}[theorem]{Proposition}
\newtheorem{corollary}[theorem]{Corollary}
\newtheorem{definition}[theorem]{Definition}
\newtheorem{maintheorem}{Theorem}
\newtheorem{remark}[theorem]{Remark}

\newtheorem{example}[theorem]{Example}
\newtheorem{question}[theorem]{Question}

\newcommand{\N}{\mathbb{N}}

\newcommand{\Z}{\mathbb{Z}}

\newcommand{\G}{G}
\newcommand{\e}{\mathbf{\mathscr{E}}_{2}}
\newcommand{\os}{\mathscr{E}_{1}}
\newcommand{\lb}{\mathbb{\llbracket}}
\newcommand{\rb}{\mathbb{\rrbracket}}
\begin{document}

\title{Infinite Eulerian paths are computable on graphs with vertices of infinite degree}
\author{Nicanor Carrasco-Vargas}
\date{}
\maketitle
\begin{abstract}
The Erdős, Grünwald, and Weiszfeld theorem is a characterization of those
infinite graphs which are Eulerian. That is, infinite graphs that
admit infinite Eulerian paths. In this article we prove an effective
version of the Erdős, Grünwald, and Weiszfeld theorem for a class
of graphs where vertices of infinite degree are allowed, generalizing
a theorem of D.Bean. Our results are obtained from a characterization
of those finite paths in a graph that can be extended to infinite
Eulerian paths.
\end{abstract}	\section{Introduction}

A basic notion in graph theory is that of an Eulerian circuit, which
is a circuit that visits every edge exactly once. It is very easy
to determine whether a finite graph admits an Eulerian circuit: by
Euler's theorem, we just need to check the parity of the vertex degrees.
It follows from Euler's theorem that computing Eulerian circuits is
also easy from the computational point of view, requiring linear time
on the number of edges \cite{fleischner_Eulerian_1990}. This shows
some contrast with other classic objects such as Hamiltonian paths
or 3-colorings, whose existence does not have a known simple characterization,
and whose computation is believed to be impossible in polynomial time
(it is an NP-complete problem). 

The Erdős, Grünwald, and Weiszfeld theorem generalizes Euler's theorem
to infinite graphs, that is, this result characterizes which infinite
graphs admit infinite Eulerian paths \cite{erdos_Eulerian_1936}.
According to this result, the existence of infinite Eulerian paths
is determined by the parity of the vertex degrees, and by the number
of \emph{ends }of the graph. 

As in the finite case, infinite Eulerian paths are comparatively easy
to compute. The fundamental result is due to D.Bean, who proved that
the Erdős, Grünwald, and Weiszfeld theorem is effective for highly
computable graphs \cite{bean_Effective_1976}. This means that, given
a highly computable graph which satisfies the hypothesis of the theorem,
it is algorithmically possible to compute an infinite Eulerian path.
In other words, for highly computable graphs the existence of an infinite
Eulerian path is equivalent to the existence of a computable one. 

The Erdős, Grünwald, and Weiszfeld theorem is one of the few classical
results in the theory of infinite graphs that is effective for highly
computable graphs. There are some results on colorings of graphs which
are effective for highly computable graphs \cite{zbMATH03749036,zbMATH03611389,TVERBERG198427,zbMATH03836015}.
On the other hand, non effective results include König's infinity
Lemma, Hall's matching theorem \cite{manaster_Effective_1972}, Ramsey's
theorem \cite{specker_Ramsey_1971,jockusch_Ramsey_1972}, certain
results regarding edge and vertex colorings \cite{zbMATH03611389,bean_Effective_1976,remmel_Graph_1986},
and more recently, domatic partitions \cite{jura_Domatic_2014}. These
results lie in the context of Nerode’s recursive mathematics program,
where one takes theorems whose classical proofs are not effective
and either (1) finds an alternative proof that is effective, (2) finds
a weaker statement that is effective, or (3) proves that the result
cannot be made effective. 

A graph is computable when the relations of adjacency and incidence
are decidable, and highly computable when it is computable, locally
finite, and with a computable vertex degree function. D.Bean also
proved that the Erdős, Grünwald, and Weiszfeld theorem cannot be made
effective for computable graphs. That is, there are computable graphs
which admit infinite Eulerian paths, but all of them are uncomputable
\cite{bean_Effective_1976}. We mention that the Erdős, Grünwald,
and Weiszfeld theorem has been revisited in recent years \cite{zbMATH05725868,inproceedings,jura_Acomputable_2016}
from the perspective of computability theory, and has been used to
compare different computability notions for infinite graphs. 

Now we provide a brief account of our results, which require rather
long definitions. In the present article we consider a class of graphs
which lie in between computable graphs and highly computable graphs.
More specifically, we weaken the hypothesis of highly computable graph
by dropping the local finiteness hypothesis, so the graph may have
vertices with infinite degree. However, we still require the vertex
degree function
\[
V(G)\to\N\cup\{\infty\}
\]
to be computable. We call these \emph{moderately computable graphs}.
Our main result is that the Erdős, Grünwald, and Weiszfeld theorem
is effective for moderately computable graphs. Thus, a moderately
computable graph admits an infinite Eulerian path if and only if it
admits a computable one. To our knowledge, this is the first positive
computability result for graphs with vertices of infinite degree. 

For graphs satisfying the corresponding hypotheses, we show something
stronger than the existence of a computable infinite Eulerian path:
we prove that it is algorithmically decidable whether a finite path
can be extended to an infinite Eulerian path.

Our results are obtained by characterizing, in a graph admitting infinite
Eulerian paths, the class of finite paths that can be extended to
some infinite Eulerian path. This characterization provides a relatively
short and independent proof of the Erdős, Grünwald, and Weiszfeld
theorem.

When considering infinite Eulerian paths, we need to distinguish between
one-way infinite Eulerian paths and two-way infinite Eulerian paths.
All the previous discussion is valid for both classes of infinite
Eulerian paths. In the following section we present our results, and
introduce some definitions.

\section{Definitions and results}

In order to state our results we will use standard concepts from graph
theory, the reader is referred to the preliminaries for unexplained
notions. A particularly relevant concept for us is that of ends, so
we repeat the definition here. The \textbf{number of ends} of a graph
$\G$ is the supremum of the number of infinite connected components
that we can obtain by removing from $G$ a finite set of edges. Note
that a graph may have infinitely many ends, but here we will only
consider graphs with finitely many ends. We can now state the Erdős,
Grünwald, and Weiszfeld theorem, and the conditions involved in the
characterization:
\begin{theorem}[\cite{erdos_Eulerian_1936}]
\label{thm:egw-one-sided}\label{thm:egw-two-sided} \label{thm:egw}A
graph $\G$ admits a one-way (resp. two-way) infinite Eulerian path
if and only if it satisfies $\os$ (resp. $\e$). 
\end{theorem}

\begin{definition}
\label{def:egw-conditions-e1}$\os$ stands for the following set
of conditions for a graph $G$. 
\begin{itemize}
	\item $E(\G)$ is countable and infinite.
	\item $\G$ is connected.
	\item Either $G$ has exactly one vertex with odd degree, or $G$ has at
	least one vertex with infinite degree and no vertices with odd degree. 
	\item $\G$ has one end.
\end{itemize}
\end{definition}

\begin{definition}
$\e$ stands for the following set of conditions for a graph $G$.
\begin{itemize}
	\item $E(\G)$ is countable and infinite.
	\item $\G$ is connected.
	\item The degree of each vertex is infinite or even.
	\item $\G$ has one or two ends. Moreover, if $E$ is a finite set of edges
	which induces a subgraph where all vertices have even degree, then
	$\G-E$ has one infinite connected component.
\end{itemize}
\end{definition}

The original statement did not explicitly mention the concept of ends,
which was coined some years later \cite{hopf_Enden_1944,freudenthal_Ueber_1945,halin_Ueber_1964}
(see \cite{diestel_Graphtheoretical_2003} for a modern source). It
is interesting to note, however, that the characterization proved
by Erdős, Grünwald, and Weiszfeld anticipated this notion. 

We now state our computability results. As explained before, in the
present work we introduce the notion of moderately computable graph:
\begin{definition}
A \textbf{moderately computable graph }is a computable graph for which
the vertex degree function $V(G)\to\N\cup\{\infty\}$ is computable. 
\end{definition}
The definitions of computable, moderately computable, and highly computable
graph are reviewed in detail in \Cref{subsec:different=000020computability=000020notions=000020for=000020graphs}.
We make some further comments on the notion of moderately computable
graph in \Cref{sec:remarks}. For now, we remark that the class
of moderately computable graphs generalizes that of highly computable
graphs in the following sense: when restricted to locally finite graphs,
the two notions coincide. 

The main result of this work is that the Erdős, Grünwald, and Weiszfeld
theorem is effective for moderately computable graphs:

\begin{maintheorem}\label{thm:effective=000020egw}

If a moderately computable graph satisfies $\os$ (resp. $\e$), then
it admits a computable one-way (resp. two-way) infinite Eulerian path.

\end{maintheorem}

Observe that this result states the existence of \emph{at least one
}computable infinite Eulerian path. We prove a slightly stronger result:

\begin{maintheorem}\label{thm:right=000020bi=000020extensible=000020is=000020decidable}

In a moderately computable graph satisfying $\os$ (resp. $\e$),
it is algorithmically decidable whether a finite path can be extended
to a one-way (resp. two-way) infinite Eulerian path.

\end{maintheorem}
\begin{remark}
\Cref{thm:right=000020bi=000020extensible=000020is=000020decidable}
guarantees that a graph $G$ as in the statement has \emph{many }computable
infinite Eulerian paths, in the sense that they are dense inside the
set of all infinite Eulerian paths. We sketch an argument. Let
$\os(G)$ (resp. $\e(G)$) be the set of all infinite one-way (resp.
two-way) Eulerian paths in $G$, but coded in a canonical way as a
subset of $\N^{\N}$. Then it follows directly from \Cref{thm:right=000020bi=000020extensible=000020is=000020decidable}
that the set $\os(G)$ (resp. $\e(G)$) has the following property: the set
of words in $\N^{\ast}$ such that the associated cylinder set intersect
$\os(G)$ (resp. $\e(G)$), is decidable. But it is a general fact
that a subset of $\N^{\N}$ with this property has a dense and recursively
enumerable set of computable points (see \cite[Proposition 2.3.2]{hoyrup_Genericity_2017}).
Each one of these points corresponds to a computable infinite Eulerian
path on $G$, so the claim follows. 
\end{remark}

\begin{remark}
In \Cref{thm:effective=000020egw} and \Cref{thm:right=000020bi=000020extensible=000020is=000020decidable},
the algorithms associated to moderately computable graphs satisfying
$\e$ are uniform on the graph (\Cref{prop:decidability-bi-extensible}).
This means that the graph is given to the algorithm as input, or that
the same algorithm works for all graphs under consideration. For graphs
satisfying $\os$ this is no longer the case, and our algorithms are
uniform only for highly computable graphs (\Cref{prop:decidability-right-extensible}).
The reader is referred to \Cref{subsec:Decidability-of-path-properties}
for details. We also mention that the proof of \Cref{thm:effective=000020egw}
for highly computable graphs given by D.Bean is also uniform.
\end{remark}

Our computability results are obtained by characterizing those finite
paths in a graph that can be extended to one-way and two-way infinite
Eulerian paths. We will now state this characterization, which may
be of independent interest. This result considers the following property
for vertices: 
\begin{definition}
\label{def:distinguished} Let $G$ be a graph satisfying $\os$.
The set of \textbf{distinguished }vertices of $G$ is the following.
If $G$ has one vertex with odd degree, then this is its only distinguished
vertex. If $G$ does not have one vertex with odd degree, then all
vertices with infinite degree in $G$ are distinguished.
\end{definition}

In the following two statements, $G-t$ denotes the graph obtained
by removing from $G$ all edges visited by the finite path $t$, and
then removing all vertices left with no incident edge (see also the
preliminaries).

\begin{maintheorem}\label{thm:e1-characterization-of-trails}

Let $\G$ be a graph satisfying $\os$. Then a finite path $t$ can
be extended to a one-way infinite Eulerian path if and only if it
satisfies the following three conditions:
\begin{enumerate}
	\item $\G-t$ is connected. 
	\item The initial vertex of $t$ is distinguished in $\G$.
	\item There is an edge $e$ incident to the final vertex of $t$ which was
	not visited by $t$. 
\end{enumerate}
Moreover, a vertex is the initial vertex of a one-way infinite Eulerian
path if and only if it is distinguished in $G$. 

\end{maintheorem}

\begin{maintheorem}\label{thm:e2-characterization-of-trails}

Let $\G$ be a graph satisfying $\e$. Then a finite path $t$ can
be extended to a two-way infinite Eulerian path if and only if it
satisfies the following three conditions:
\begin{enumerate}
	\item $\G-t$ has no finite connected components.
	\item There is an edge $e$ incident to the final vertex of $t$ which was
	not visited by $t$. 
	\item There is an edge $f\ne e$ incident to the initial vertex of $t$
	which was not visited by $t$. 
\end{enumerate}
\end{maintheorem}

These results also provide a relatively short proof of the Erdős,
Grünwald, and Weiszfeld theorem (\Cref{cor:egw-one-sided} and
\Cref{cor:egw-two-sided}). Naturally, our proof of \Cref{thm:right=000020bi=000020extensible=000020is=000020decidable}
goes by showing that our characterization for finite paths is decidable
when the graph is moderately computable. This is not at all trivial
when the graph has two ends, and in order to show the decidability
of the corresponding property we will need to prove a second characterization
of paths that can be extended to two-way infinite Eulerian paths  (\Cref{prop:bi-extensible-alternative-characterization}).

\subsection{Paper structure}

In \Cref{sec:Preliminaries} we review some basic facts and terminology
from graph theory. In \Cref{sec:egw} we prove our characterization
of those finite paths in a graph that can be extended to infinite
Eulerian paths, that is, \Cref{thm:e1-characterization-of-trails}
and \Cref{thm:e2-characterization-of-trails}. In \Cref{sec:computability}
we review some computability notions for infinite graphs, and prove
that our characterization for finite paths is decidable on moderately
computable graphs. From this we obtain \Cref{thm:effective=000020egw}
and \Cref{thm:right=000020bi=000020extensible=000020is=000020decidable}.
Finally, we make some remarks on the notion of moderately computable
graph in \Cref{sec:remarks}.

\subsection*{Acknowledgements}
The author is grateful to Rod Downey for pointing out the results
in Bean's paper. The author is also grateful to Vittorio Cipriani,
Vallentino Delle Rose, and Crist\'obal Rojas for fruitful discussions
which derived in significant improvements from a first version of
this article, in particular concerning \Cref{prop:bi-extensible-alternative-characterization}
and the notion of moderately computable graph. The author is also
grateful to Sebasti\'an Barbieri and Paola Rivera for their very
helpful comments. The author also thanks the anonymous reviewer for
their helpful suggestions. 

This research was partially supported by ANID CONICYT-PFCHA/Doctorado Nacional/2020-21201185, ANID/Basal National Center for Artificial Intelligence CENIA FB210017, and the European Union's Horizon 2020 research and innovation program under the Marie Sklodowska-Curie grant agreement No 731143.
\section{Preliminaries}\label{sec:Preliminaries}

\subsection{Graph theory}

Here we fix some graph theory terminology, and recall some well known
facts. The reader is referred to the books \cite{diestel_Graph_2017,bollobas_Modern_1998}.
For ends on infinite graphs, the reader is referred to \cite{diestel_Graphtheoretical_2003}. 

Throughout this paper we deal with finite and infinite undirected
graphs, where two vertices may be joined by multiple edges, and loops
are allowed. The vertex set of a graph $\G$ is denoted by $V(\G)$,
and its edge set by $E(\G)$. Each edge \textbf{joins} a pair of
vertices, and is said to be \textbf{incident} to these vertices. Two
vertices joined by an edge are called \textbf{adjacent}. A \textbf{loop}
is an edge joining a vertex to itself. The \textbf{degree} $\deg_{\G}(v)$
of the vertex $v$ is the number of edges in $G$ incident to $v$,
where loops are counted twice. 

It will be convenient for us to define paths as graph homomorphisms.
\textbf{A graph homomorphism} $f\colon\G'\to\G$ is a function which
sends vertices to vertices, edges to edges, and is compatible with
the incidence relation. 

We denote by $\llbracket a,b\rrbracket$ the graph with vertex set
$\{a,\dots,b\}\subset\Z$, and with edges $\{c,c+1\}$ joining $c$
to $c+1$, for $c\in\{a,\dots,b-1\}$. The graphs $\lb\N\rb$ and
$\lb\Z\rb$ are defined in a similar manner. We assume that $0$ belongs
to $\N$. 

A \textbf{path} (resp. \textbf{one-way infinite path}, \textbf{two-way
infinite path}) on $\G$ is a graph homomorphism $t\colon\lb a,b\rb\to\G$
(resp. $t\colon\lb\N\rb\to\G$, $t\colon\lb\Z\rb\to\G$) which does
not repeat edges. In these cases we say that $t$ \textbf{visits}
the vertices and edges in its image, and we call it \textbf{Eulerian}
when it visits every edge of $\G$ exactly once. Note that a path
is finite by definition, but we may emphasize this fact by writing
``finite path''.

Let $t\colon\lb a,b\rb\to\G$ be a path. We say that $t(a)$ is its
\textbf{initial vertex }and $t(b)$ is its \textbf{final vertex}.
We say that $t$ \textbf{joins} $t(a)$ to $t(b)$. A path is called
a \textbf{circuit} when its initial vertex is equal to its final vertex. 

The graph $G$ is said to be \textbf{connected} when every pair of
vertices is joined by a path. In this case, the vertex set $V(G)$
becomes a metric space. The distance $d_{G}$ between two vertices
is the length of the shortest path joining them, where the length
of a path is the number of edges that it visits. The distance from
a vertex to itself is defined as zero. 

In several constructions, we will consider subgraphs induced by sets
of edges. Given a set of edges $E\subset E(\G)$, the \textbf{ induced subgraph } $\G[E]$ is defined as follows. The edge set of $\G[E]$ is $E$, and its vertex set is the 
set of all vertices incident some edge in $E$. Given a set of edges
$E\subset E(\G)$, we denote by $\G-E$ the graph obtained by removing
from $G$ all edges in $E$, and then removing all vertices left with
no incident edge. In other words, $\G-E$ is the subgraph of $G$
induced by the set of edges $E(G)-E$. Given a path $t$, $\G-t$
denotes the graph obtained by removing from $G$ all edges visited
by $t$, and then removing all vertices left with no incident edge.
We say that $\G-t$ is obtained by \textbf{removing }the path $t$
from $\G$. We denote by $\deg_{t}(v)$ the degree of $v$ in the
subgraph of $G$ induced by the set of edges visited by $t$. 

A graph is \textbf{finite }if its edge set is finite, and \textbf{locally
finite} if every vertex has finite degree. A graph is \textbf{simple
}if it has no loops and every pair of vertices is joined by at most
one edge. A graph is said to be \textbf{even} if every vertex has
finite and even degree. A \textbf{connected component }in $\G$ is
a connected subgraph of $\G$ which is maximal for the subgraph relation.
The \textbf{number of ends} of the graph $\G$ is the supremum of
the number of infinite connected components that we can obtain by
removing from $G$ a finite set of edges.

We recall that Euler's theorem\textbf{ }asserts that a finite and
connected graph $\G$ admits an Eulerian circuit if and only if all
vertices have even degree, and an Eulerian path from a vertex $u$
to $v\ne u$ if and only if these are the only vertices in $G$ with
odd degree. The Handshaking lemma\textbf{ }asserts that the sum of
all vertex degrees of a finite graph $G$ equals twice the cardinality
of $E(G)$, and in particular is an even number.

\subsection{Concatenation and inversion of paths}

We introduce some ad hoc terminology regarding ``pasting'' paths.
Let $t\colon\lb a,b\rb\to\G$ and $s\colon\lb c,d\rb\to\G$ be edge
disjoint paths. If the final vertex of $t$ is also the initial vertex
of $s$, the \textbf{concatenation of $s$ at the right} of $t$ is
the path whose domain is $\lb a,b+d-c\rb$, whose restriction to $\lb a,b\rb$
equals $t$, and whose restriction to $\lb b,b+d-c\rb$ follows the
same path as $s$, but with the domain shifted. %
\begin{comment}
\begin{align*}
	r(x) & =\begin{cases}
		t(x) & x\in[a,b]\\
		s(x+c-b) & x\in[b,b+d]
	\end{cases}\\
	r(\{x,x+1\}) & =\begin{cases}
		t(\{x,x+1\}) & x\in[a,b-1]\\
		s(\{x+c-b,x+c-b+1\}) & x\in[b,b+d-1]
	\end{cases}
\end{align*}
\end{comment}
{} If the final vertex of $s$ coincides with the initial vertex of
$t$, we define the \textbf{concatenation of $s$ at the left of $t$
}as the path whose domain is $\lb a-(d-c),b\rb$, which on $\lb a-(d-c),a\rb$
follows the path of $s$ but with the domain shifted, and on $\lb a,b\rb$
follows the path of $t$. 

We say that a finite or infinite path \textbf{extends} the finite
path $t$ if its restriction to the domain of $t$ is equal to $t$.
For example, if we concatenate a path at the right or left of $t$,
we obtain a path that extends $t$. Finally, we define the \textbf{inverse
}of $t$, denoted $-t$, as the path with domain $\lb-b,-a\rb$ and
which visits the vertices and edges visited by $t$ in but in inverse
order.

\section{A characterization of finite paths which can be extended to infinite
Eulerian paths}\label{sec:egw}

In this section we prove our characterization of those finite paths
in a graph that can be extended to one-way and two-way infinite Eulerian
paths. That is, \Cref{thm:e1-characterization-of-trails} and
\Cref{thm:e2-characterization-of-trails}. We start by proving
a technical result. 
\begin{lemma}
\label{lem:technical-lemma}Let $\G$ be an infinite and connected
graph, and let $t$ be a path on $\G$ such that every vertex in $G$
different from the initial and final vertex of $t$ has either even
or infinite degree in $\G$. Then there is a path that  visits all
vertices and edges visited by $t$, with the same initial and final
vertices as $t$, and whose removal from $\G$ leaves no finite connected
component. 
\end{lemma}

\begin{proof}
Let $\G'$ be the subgraph of $\G$ induced by the edges visited by
$t$ and the edges in finite connected components of $\G-t$. Observe
that $G'$ is finite as there is at most one finite connected components
in $G-t$ for each vertex visited by $t$. It is clear that $G'$
is connected, and that $G-E(G')$ has no finite connected component.
\begin{comment}
	To see that it is finite, note that each vertex visited by $t$ belongs
	to at most one connected component of $\G-t$, and there is at most
	one finite connected component of $\G-t$ for each vertex visited
	by $t$.
\end{comment}
{} 

Let $u$ be the initial vertex of $t$, and let $v$ be the final
vertex of $t$. We will show that all vertices in $\G'$ different
from $u$ or $v$ have even degree in $\G'$. Then we will show that
$u$ and $v$ have both even degree in $\G'$ if they are equal, and
otherwise both have odd degree in $\G'$. This proves the statement,
as it suffices to apply Euler's theorem to the graph $G'$. Indeed,
an Eulerian path on $G'$ (which could be a circuit) has the 
properties mentioned in the statement. 

Let $w$ be a vertex in $\G'$ not visited by $t$, so $\deg_{G}(w)=\deg_{G'}(w)$.
As $\G'$ is a finite graph it follows that $\deg_{\G}(w)$ is finite,
and then the hypothesis on $G$ implies that $\deg_{G'}(w)$ is an
even number. Now let $w$ be a vertex visited by $t$, but different
from $u$ and $v$. We verify that $\deg_{\G'}(w)$ is finite and
even. Indeed, there is at most one connected component of $\G-t$
containing $w$. If this connected component is infinite, then there
is no finite connected component of $\G-t$ containing $w$, and thus
$\deg_{\G'}(w)=\deg_{t}(w)$, a finite and even number. If this connected
component is finite then it follows that $\deg_{\G}(v)$ is finite,
and thus it is even by our hypothesis on $G$. It follows that $\deg_{\G'}(v)$
is also even. We finally consider $u$ and $v$. If our claim on the
degrees of $u$ and $v$ on $G'$ fails then $\G'$ would have exactly
one vertex with odd degree, contradicting the Handshaking lemma. As
explained in the previous paragraph, the statement now follows by
applying Euler's theorem to the graph $G'$. 
\end{proof}

\subsection{The case of one-way infinite paths}

In this subsection, we prove \Cref{thm:e1-characterization-of-trails}.
For this purpose we introduce the following definition:
\begin{definition}
\label{def:right-extensible} Let $\G$ be a graph satisfying $\os$.
We say that a path $t$ is \textbf{right-extensible} in $\G$ if the
following three conditions hold:
\begin{enumerate}
	\item $\G-t$ is connected. 
	\item The initial vertex of $t$ is distinguished in $\G$.
	\item There is an edge $e$ incident to the final vertex of $t$ which was
	not visited by $t$. 
\end{enumerate}
\end{definition}

\Cref{thm:e1-characterization-of-trails} asserts that a path
can be extended to a one-way infinite Eulerian path if and only if
it is right-extensible. The structure of our proof is very simple,
we show that right-extensible paths exist, and that they can be extended
to larger right-extensible paths. This allows an iterative construction.
A key idea is that the removal of a right-extensible path from a graph
satisfying $\os$ leaves a graph satisfying $\os$. 
\begin{lemma}
\label{lem:borrado-right-extensible}Let $\G$ be a graph satisfying
$\os$. If $t$ is a right-extensible path in $\G$, then $\G-t$
also satisfies $\os$. Moreover, the final vertex of $t$ is distinguished
in $\G-t$.
\end{lemma}

\begin{proof}
The proof of this result is a case by case review of the vertex degree
of the initial and final vertex of $t$, in both $G$ and $G-t$. 
\end{proof}

\begin{lemma}
\label{lem:existencia-right-extensible}Let $\G$ be a graph satisfying
$\os$, and let $v$ be a distinguished vertex in $\G$. Then for
every edge $e$ there is a right-extensible path on $\G$ whose initial
vertex is $v$ and which visits $e$. 
\end{lemma}

\begin{proof}
As $\G$ is connected there is a path $t\colon\lb0,c\rb\to\G$ with
$t(0)=v$ and which visits $e$. By \Cref{lem:technical-lemma}
we can assume that the removal of $t$ leaves no finite connected
component, and thus $G-t$ is a connected graph. 

We claim that $t$ is right-extensible. The first and second conditions
in the definition hold by our choice of $t$. For the third condition
we separate the cases where $t$ is a circuit or not. If $t$ is a
circuit then $\deg_{t}(t(b))$ is even while $\deg_{G}(t(b))$ is
either odd or infinite, as $t(0)=t(b)$ is distinguished. If $t$
is not a circuit then $\deg_{t}(t(b))$ is odd while $\deg_{\G}(t(b))$
is either even or infinite. In both cases it follows that $t(b)$
has edges in $\G$ not visited by $t$, so the third condition in
the definition of right-extensible is also satisfied. 
\end{proof}
\begin{lemma}
\label{lem:existencia-extension-right-extensible}Let $\G$ be a graph
satisfying $\os$. Then for every right-extensible path $t$ and edge
$e$ there is a path which is right-extensible on $G$, extends $t$,
and visits $e$. 
\end{lemma}

\begin{proof}
By \Cref{lem:borrado-right-extensible} the graph $\G-t$ satisfies
$\os$ and contains the final vertex of $t$ as a distinguished vertex.
Now by \Cref{lem:existencia-right-extensible} the graph $\G-t$
admits a right-extensible path $s:\lb0,c\rb\to\G-t$ which starts
at the final vertex of $t$, and which visits $e$. Thus the path
$t':\lb0,b+c\rb\to\G$ obtained by concatenating $s$ at the right
of $t$ is right-extensible in $\G$, visits $e$, and extends $t$. 
\end{proof}
We are now ready to prove \Cref{thm:e1-characterization-of-trails}. 

\begin{proposition}[\Cref{thm:e1-characterization-of-trails}]
\label{prop:extensibility-one-sided}Let $\G$ be a graph satisfying
$\os$. Then a path on $\G$ is right-extensible if and only if it
can be extended to a one-way infinite Eulerian path. Moreover, a vertex
is distinguished in $\G$ if and only if it is the initial vertex
of a one-way infinite Eulerian path on $\G$.
\end{proposition}

\begin{proof}
We first prove that an arbitrary path $t$ which is right-extensible
on $G$, can be extended to a one-way infinite Eulerian path on $G$.
Let $(e_{n})_{n\in\N}$ be a numbering of $E(G)$. We iterate \Cref{lem:existencia-extension-right-extensible}
to obtain a sequence of paths $(t_{n})_{n\in\N}$ such that $t_{0}$
extends $t$, and such that for each $n$:
\begin{enumerate}
	\item The path $t_{n}$ is right-extensible.
	\item The path $t_{n}$ visits $e_{n}$.
	\item The path $t_{n+1}$ extends $t_{n}$. 
\end{enumerate}
This sequence defines a one-way infinite Eulerian path on $G$ which
extends $t$. 

Now let $v$ be a distinguished vertex. Then \Cref{lem:existencia-right-extensible}
shows that it is the initial vertex of a right-extensible path, and we already proved that such a path can be extended to a one-way infinite Eulerian path. It follows that $v$ is the initial vertex of a one-way infinite Eulerian path. 

For the remaining implications, let $T$ be a one-way infinite Eulerian
path on $\G$. We claim that $T(0)$ is distinguished. Indeed, it
is clear that $T(0)$ can not have finite even degree, and that if
$T(0)$ has finite degree then every other vertex has even or infinite
degree. That is, the initial vertex of $T$ is distinguished. Moreover, it is clear from the definition that the restriction
of $T$ to a domain of the form $\lb0,n\rb$ is a right-extensible
path.
\end{proof}
As mentioned in the introduction, we obtain the Erdős, Grünwald, and
Weiszfeld theorem as a corollary of our results about paths. 
\begin{corollary}[Erdős, Grünwald, and Weiszfeld theorem for one-way infinite Eulerian
paths]
\label{cor:egw-one-sided}A graph admits a one-way infinite Eulerian
path if and only if it satisfies $\os$. 
\end{corollary}

\begin{proof}
If a graph satisfies $\os$, then it has a distinguished vertex. It
follows from \Cref{prop:extensibility-one-sided} that it admits
a one-way infinite Eulerian path. 

Now let $G$ be a graph which admits a one-way infinite Eulerian path
$T$. The existence of this path shows that $E(G)$ is countable,
infinite, and that $G$ is connected. 

Let us prove that $G$ has one end. Let $E$ be a finite set of edges.
We must prove that $G-E$ has one infinite connected component. For
this purpose, consider the set of edges visited by $T$ after visiting
all edges from $E$. The infinite path $T$ shows that these edges
constitute a connected component of $G-T$. Moreover, it must be the
only one because this set of edges has finite complement in $E(G)$.
This proves our claim that $G$ has one end.

Finally, we prove the claim about the parity of the degrees of the
vertices of $G$. Indeed, a counting argument considering the number
of times that $T$ goes ``in'' and ``out'' each vertex shows that
the initial vertex of $T$ must have either odd or infinite degree,
and every other vertex must have either even or infinite degree. 
\end{proof}

\subsection{The case of two-way infinite paths}

In this subsection, we prove \Cref{thm:e2-characterization-of-trails}.
For this purpose, we introduce the following definition:
\begin{definition}
\label{def:bi-extensible}Let $\G$ be a graph satisfying $\e$. We
say that a path $t$ is \textbf{bi-extensible} in \textbf{$\G$} if
the following three conditions hold:
\begin{enumerate}
	\item $\G-t$ has no finite connected components.
	\item There is an edge $e$ incident to the final vertex of $t$ which was
	not visited by $t$. 
	\item There is an edge $f\ne e$ incident to the initial vertex of $t$
	which was not visited by $t$. 
\end{enumerate}
\end{definition}

\Cref{thm:e2-characterization-of-trails} asserts that a path
can be extended to a two-way infinite Eulerian path if and only if
it is bi-extensible. The structure of our proof is very similar to
the proof given for right-extensible paths. A key idea is that the
removal of a bi-extensible \emph{circuit} from a graph satisfying
$\e$ leaves a graph satisfying $\e$. 
\begin{lemma}
\label{lem:erase-bi-extensible-closed-trails}Let $\G$ be a graph
satisfying $\e$, and let $t$ be a bi-extensible circuit. Then $\G-t$
also satisfies $\e$.
\end{lemma}

\begin{proof}
Note that $G-t$ is connected. Indeed as $t$ is bi-extensible, $G-t$
has no finite connected components, and by the third condition in
$\e$ the graph $G-t$ has at most one infinite connected component.
Joining these two facts, we conclude that $G-t$ is connected. The
remaining conditions in $\e$ are easily verified. 
\end{proof}
\begin{lemma}
\label{lem:existence-bi-extensible-trails} Let $\G$ be a graph satisfying
$\e$. Then for every vertex $v$ and edge $e$ there is a bi-extensible
path which visits $v$ and $e$.
\end{lemma}

\begin{proof}
By connectedness of $\G$ there is a path $t:\lb0,b\rb\to\G$ which
visits both $v$ and $e$. By \Cref{lem:technical-lemma} we
can assume that the removal of this path leaves no finite connected
component in $\G$. We now consider two cases: 

In the first case $t$ is not a circuit. We assert that then it is
bi-extensible. Indeed, as $t(0)$ and $t(b)$ have odd degree in $t$,
there are edges $e$ and $f$ not visited by $t$, and incident to
$t(0)$ and $t(b)$ respectively. We can take $e\ne f$ because otherwise
the graph $\G-t$ would be forced to have a finite connected component,
contradicting our choice of $t$. 

In the second case $t$ is a circuit. We just have to re-parametrize this path
in order to obtain a bi-extensible path. Indeed, as $\G$ is connected, $t$
visits a vertex $u$ which lies in $\G-t$. As the degree of $u$
is even both in $\G-t$ and $G$, there are at least two different edges in $\G-t$
which are incident to $u$. This holds even if some edge incident
to $u$ in $\G-t$ is a loop. Now we simply re-parametrize $t$ so that
it becomes a circuit with $u$ as initial and final vertex. It is
clear that this new path is bi-extensible. 
\end{proof}
\begin{lemma}
\label{lem:extension-bi-extensible-trails}Let $\G$ be a graph satisfying
$\e$. Then for every bi-extensible path $t$ and edge $e$ there is
a bi-extensible path which extends $t$ and visits $e$. We can choose
this extension of $t$ so that its domain strictly extends the domain
of $t$ in both directions.
\end{lemma}

\begin{proof}
Let $t$ and $e$ be as in the statement. Clearly it suffices to prove
the existence of a bi-extensible path $s$ which extends $t$, visits
$e$, and whose domain strictly extends the domain of $t$ in a direction
of our choice. In order to prove this claim we consider three cases:

In the first case, $t\colon\lb a,b\rb\to\G$ is a circuit. Then $\G-t$
is connected by the third condition in $\e$, and the graph $\G-t$
satisfies $\e$ by \Cref{lem:erase-bi-extensible-closed-trails}.
We apply \Cref{lem:existence-bi-extensible-trails} to the graph
$\G-t$ to obtain a path $t_{1}\colon\lb a_{1},b_{1}\rb\to\G-t$ which
visits the vertex $t(a)=t(b)$, the edge $e$, and is bi-extensible
on $\G-t$. We just need to split $t_{1}$ in two paths and concatenate
them to $t$. For this let $c_{1}\in\N$, $a_{1}\leq c_1 \leq b_{1}$, such that $t_{1}(c_{1})=t(a)$,
and define $l_{1}$ and $r_{1}$ as the restrictions of $t_{1}$ to
$\lb a_{1},c_{1}\rb$ and $\lb c_{1},b_{1}\rb$, respectively. We
define the path $s$ by concatenating $l_{1}$ to the left of $t$,
and then $r_{1}$ to its right. By our choice of $t_{1}$ and $c_{1}$,
it follows that $s$ visits $e$, extends $t$, and is bi-extensible
on $\G$. An alternative way to define the path $s$ is by concatenating
$-r_{1}$ at the left of $t$, and $-l_{1}$ to its right. 

Observe that it is possible that $c_{1}$ equals $a_{1}$ or $b_{1}$,
and in this situation the domain of $s$ extends that of $t$ only
in one direction. But can choose this direction with the two possible
definitions of $s$. 

In the second case, $t\colon\lb a,b\rb\to\G$ is not a circuit and
$\G-t$ is connected. In this case we prove that we can extend $t$
to a bi-extensible circuit $s$, this proves the claim as then we
go back to the first case. As $\G-t$ is connected we can take a path
$t_{2}\colon\lb a_{2},b_{2}\rb\to\G-t$ whose initial vertex is $t(a)$
and whose final vertex is $t(b)$. By \Cref{lem:technical-lemma}
we can assume that $(\G-t)-t_{2}$ has no finite connected components.
We just need to split $t_{2}$ in two paths and concatenate them to
$t$ as follows. Let $c_{2}\in\N$, $a_{2}\leq c_2 \leq b_{2}$, such that $t_{2}(c_{2})$
lies in $(\G-t)-t_{2}$. That is, $t_{2}(c_{2})$ is a vertex with
incident edges not visited by $t_{2}$. Observe that there must be
at least two such edges by the parity of the vertex degrees. We define
$l_{2}$ and $r_{2}$ as the restrictions of $t_{2}$ to $\lb a_{2},c_{2}\rb$
and $\lb c_{2},b_{2}\rb$, respectively. Note that the final vertex
of $-l_{2}$ is $t(a)$, and the initial vertex of $-r_{2}$ is $t(b)$.
We define $s$ by concatenating $-l_{2}$ to the left of $t$, and
then $-r_{2}$ to its right. By our choice of $t_{2}$ and $c_{2}$,
$s$ is a bi-extensible circuit which extends $t$. This concludes
our proof of the second case. 

In the third case, $t$ is not a circuit and $\G-t$ is not connected.
It follows that $\G-t$ has two infinite connected components and
no finite connected component. These components satisfy $\os$, and
have the initial and final vertex of $t$ as a distinguished vertex.
We simply apply \Cref{lem:existencia-right-extensible} on each
one of these components and then concatenate to extend $t$ as in
our claim. %

\end{proof}
We are now ready to prove \Cref{thm:e2-characterization-of-trails}.

\begin{proposition}[\Cref{thm:e2-characterization-of-trails}]
\label{prop:proposition-bi-extensibility}Let $\G$ be a graph satisfying
$\e$. Then a path is bi-extensible if and only if it can be extended
to a two-way infinite Eulerian path.
\end{proposition}

\begin{proof}
Let $t$ be a bi-extensible path, and let $(e_{n})_{n\in\N}$ be a
numbering of the edges in $E(\G)$. In order to extend $t$ to a two-way
infinite Eulerian path on $G$, we iterate \Cref{lem:extension-bi-extensible-trails}
to obtain a sequence of paths $(t_{n})_{n\in\N}$ such that $t_{0}$
extends $t$, and such that for each $n$:
\begin{enumerate}
	\item The path $t_{n}$ is bi-extensible.
	\item The path $t_{n}$ visits $e_{n}$.
	\item The path $t_{n+1}$ extends $t_{n}$. 
\end{enumerate}
This sequence defines a two-way infinite Eulerian path on $\G$ which
extends $t$.

For the other direction, it is clear that the restriction of a two-way
infinite Eulerian path $T:\lb\Z\rb\to G$ to a set of the form $\lb n,m\rb$
yields a bi-extensible path. 
\end{proof}
We now prove an alternative characterization of paths that can be
extended to two-way infinite Eulerian paths. This will be relevant
in \Cref{sec:computability} to prove \Cref{thm:right=000020bi=000020extensible=000020is=000020decidable}.
\begin{proposition}
\label{prop:bi-extensible-alternative-characterization}Let $\G$
be a graph satisfying $\e$. Then a path $t$ can be extended to a
two-way infinite Eulerian path on $\G$ if and only if it satisfies
the following three conditions:
\begin{enumerate}
	\item Every connected component of $\G-t$ contains either the initial or
	final vertex of $t$.
	\item There is an edge $e$ incident to the final vertex of $t$ which was
	not visited by $t$. 
	\item There is an edge $f\ne e$ incident to the initial vertex of $t$
	which was not visited by $t$. 
\end{enumerate}
\end{proposition}

\begin{proof}
By \Cref{prop:proposition-bi-extensibility}, it suffices to
prove that a path satisfies these conditions if and only if it is
bi-extensible. 

Let $t$ be a path satisfying these conditions, we claim that $\G-t$
has no finite connected components and thus $t$ is bi-extensible.
This is clear if $t$ is a circuit. If $t$ is not a circuit, we let
$\G_{-}$ (resp. $\G_{+}$) be the connected component of $\G-t$
containing the initial (resp. final) vertex of $t$. If $G_{+}$ is
equal to $G_{-}$, it follows that $G-t$ has no finite connected
component. If $G_{+}$ is different to $G_{-}$, we claim that both
must be infinite, and thus $\G-t$ has no finite connected component.
Indeed, as $G$ satisfies $\e$, the initial and final vertex
of $t$ are the only vertices in $\G-t$ with odd degree. Now if $\G_{+}$
(or $G_{-}$) is finite, then it has exactly one vertex with odd degree,
a contradiction by Handshaking lemma. Our claim that $t$ is bi-extensible
follows. 

Now let $t$ be a bi-extensible path, we claim that it satisfies the
three conditions in this statement. We just need to verify that every
connected component of $\G-t$ contains either the initial or final
vertex of $t$. Indeed, the fact that $t$ can be extended to a two-way infinite 
Eulerian path (\Cref{prop:proposition-bi-extensibility}) implies
that every vertex in $G-t$ can be joined by a path to the initial
vertex of $t$, or to the final vertex of $t$, so it follows that
every connected component of $\G-t$ contains either the initial or
final vertex of $t$. 
\end{proof}
As mentioned in the introduction, we obtain the Erdős, Grünwald, and
Weiszfeld theorem as a corollary of our results about paths. 
\begin{corollary}[Erdős, Grünwald, and Weiszfeld theorem for two-way infinite Eulerian
paths]
\label{cor:egw-two-sided}A graph admits a two-way infinite Eulerian
path if and only if it satisfies $\e$. 
\end{corollary}

\begin{proof}
If a graph satisfies $\e$, then it admits a two-way infinite Eulerian
path by \Cref{lem:existence-bi-extensible-trails} and \Cref{prop:proposition-bi-extensibility}. 

Now let $G$ be a graph which admits a two-way infinite Eulerian path
$T$, we verify that $G$ satisfies the conditions $\e$. The path
$T$ shows that the edge set of $G$ is infinite, countable, and that
$G$ is connected. Moreover, a counting argument on the number of
times that $T$ goes ``in'' and ``out'' a vertex shows that no
vertex can have finite odd degree. That is, every vertex has either
infinite or even degree. 

To see that $G$ has either one or two ends, note that the path $T$
shows that the removal of a finite set of edges from $G$ must leave
at most two infinite connected components. Finally, let $E$ be a
finite set of edges such that $G[E]$ is an even subgraph of $G$. We must
show that $G-E$ has one infinite connected component. For this purpose,
let $u$ and $v$ be the first and last vertex in $\G[E]$ visited
by $T$. If $u=v$ then our claim follows. If $u\ne v$ then we let
$F$ be the set of edges visited by $T$ after $u$ but before $v$,
so $E\subset F$. Observe that $\G[F]-E$ is a graph where $u$ and
$v$ have odd degree, and the remaining vertices have even degree.
If $G[F]-E$ is connected then our claim that $G-E$ has one infinite connected component follows. But this is the
only possibility: if $\G[F]-E$ is not connected, then the connected
component of $\G[F]-E$ which contains $u$ is a finite and connected
graph with exactly one vertex with odd degree, contradicting the Handshaking
lemma. 
\end{proof}

\section{Computability results  }\label{sec:computability}

In this section we prove our computability results associated to infinite
Eulerian paths on moderately computable graphs. That is, \Cref{thm:effective=000020egw}
and \Cref{thm:right=000020bi=000020extensible=000020is=000020decidable}.
We assume familiarity with basic concepts in computability or recursion
theory, such as decidable set of natural numbers and computable function
on natural numbers. By algorithm we refer to the formal object of
Turing machine. In some cases we consider algorithms dealing with
finite objects such as finite graphs
or finite sets of vertices. These objects will be assumed to be described
by natural numbers in a canonical way. A more formal treatment of
these computations can be done through the theory of numberings, but
we avoid this level of detail here. 

We start with a review of standard computability notions for infinite
graphs, and introducing the notion of moderately computable graph. 

\subsection{Computable, moderately computable, and highly computable graphs}\label{subsec:different=000020computability=000020notions=000020for=000020graphs}

Let us recall that the adjacency relation of a graph $G$ is the set
\begin{align*}
\{(u,v) & \in V(G)^{2}\mid\text{\ensuremath{u} and \ensuremath{v} are adjacent}\}.
\end{align*}

A common definition of computable graph is a graph whose vertex set
is a decidable subset of $\N$, and whose adjacency relation is a
decidable subset of $\N^{2}$. This definition is only suitable for
simple graphs, that is, graphs with no loops or multiple edges between
vertices. In order to avoid this limitation, we also consider the
incidence relation between edges and pairs of vertices:
\begin{align*}
\{(e,u,v) & \in E(G)\times V(G)^{2}\mid\text{\ensuremath{e} joins \ensuremath{u} to \ensuremath{v}}\}.
\end{align*}

\begin{definition}
\label{def:computable-graph}A graph is \textbf{computable }when its
edge and vertex sets are indexed by decidable sets of natural numbers,
in such a manner that the relations of adjacency and incidence between
edges and pairs of vertices are decidable. 
\end{definition}

In a computable and locally finite graph, the vertex degree function
may not be computable (see \cite{calvert_Distance_2011}). This is
the reason behind the following definition, which is standard in the
literature of recursive combinatorics.
\begin{definition}
A computable graph is\textbf{ highly computable }if it is locally
finite, and the vertex degree function $\deg_{\G}\colon V(\G)\to\N$
is computable. 
\end{definition}

Here we introduce the following class of graphs, which can be considered
as a natural generalization of highly computable graphs to graphs
with vertices of infinite degree. 

\begin{definition}
A computable graph is\textbf{ moderately computable }if the vertex
degree function $\deg_{G}\colon V(\G)\to\N\cup\{\infty\}$ is a computable
function.
\end{definition}

\begin{remark}
Let us compare these notions by looking at the problem of computing
the set of edges incident to a vertex $v$. When $G$ is a computable
graph, this set of edges is recursively enumerable because the relation
of incidence between edges and pairs of vertices is decidable, and
the algorithm is uniform on $v$. When $G$ is highly computable we
can compute this finite set, uniformly on $v$: we just have to enumerate
edges, and stop when the vertex degree function informs us that we
already enumerated all edges. When $G$ is moderately computable we
have a compromise between these two notions: we have the knowledge
of whether this set is finite or infinite, and if the set is finite
then we can compute it. If the graph is simple, these comments also
apply to the problem of computing the set of neighbors of a vertex,
this is discussed in more detail in \Cref{sec:remarks}. 
\end{remark}

We will also make use of the following notion. A \textbf{description
}of a computable graph $G$ is a tuple of algorithms deciding membership
in the index sets of vertices, edges, and the relations of adjacency
and incidence between pairs of vertices and edges. A description of
a highly computable graph and moderately computable graph specifies,
additionally, an algorithm for the computable function $v\mapsto\deg_{G}(v)$.

\subsection{Elementary operations on moderately computable graphs}

In this subsection we prove the computability of some basic operations
on moderately computable graphs. For this purpose, we introduce the
following notion:

\begin{definition}
Let $G$ be a moderately computable graph. Let $v$ be a vertex of
the graph $G$, a distance $r\in\N$, and a ``precision'' $s\in\N$.
We define $G(v,r,s)$ as the finite and connected subgraph of $G$
induced by the edges of all paths in $\G$ which visit $v$, have
length at most $r$, and which only visit edges whose index is at
most $s$.
\end{definition}

The graph $G(v,r,s)$ is essentially a ``ball'' around $v$. This
graph is finite, connected, and it can be computed from $v$, $r$
, $s$ and $G$:
\begin{lemma}
There is an algorithm which,
on input a description of a moderately computable graph $G$, and parameters $v$, $r$ and $s$, computes $G(v,r,s)$. 
\end{lemma}

\begin{proof}
We start by computing the finite set of all possible paths that can
be formed with the finite set of edges 
\[
\{e\in E(G)\mid\text{the index of \ensuremath{e} is at most \ensuremath{s}}\}.
\]
After this, we check which of these paths visit the vertex $v$ and
have length at most $r$, and define $E$ as the set of all edges
visited by these paths. Observe that the graph $G(v,r,s)$ is exactly
the subgraph of $G$ induced by $E$. As we already have computed
$E$, it only remains to compute the finite set $V$ of vertices incident
to the edges in $E$. We have obtained the edge and vertex set of
$G(v,r,s)$, and the relations of adjacency and incidence of $G(v,r,s)$
can be recovered from the description of $G$. 
\end{proof}
\begin{comment}
Let $G$ be a moderately computable graph. Given $v$ and $r$ as
before, it is decidable whether the graph $G(v,r)$ is finite.

We just sketch an algorithm. We start by checking whether $\deg_{G}(v)$
is finite (using the computability of the vertex degree function).
If it is finite, then we compute the finite set of all vertices which
are adjacent to $v$. Then we repeat the process for the vertices
obtained, and iterate this process. If we find some vertex of infinite
degree in less than $r$ steps, then we conclude that $G(v,r)$ is
infinite, and otherwise we conclude that $G(v,r)$ is finite.
\begin{fact}
	Let $G$ be a moderately computable graph. Given $v$, $r$, and $s$
	as before, and a vertex $u$ in $G(v,r,s)$, it is decidable whether
	all edges which are incident to $u$ in $G$ are present in $G(v,r,s)$.
\end{fact}

To prove this claim, note that an algorithm just needs to compare
$\deg_{G}(u)$ with $\deg_{G(v,r,s)}(u)$. A simple consequence of
this fact is the following:

\begin{fact}
	Let $G$ be a moderately computable graph. Given $v$, $r$, and $s$
	as before, it is decidable whether $G(v,r)$ is equal to $G(v,r,s)$.
\end{fact}

Joining these procedures we obtain an algorithm which on input $v,r$
decides whether $G(v,r)$ is finite, and if it is a finite graph,
outputs the graph $G(v,r)$. In this manner we recover the following
standard fact:
\begin{fact}
	Let $G$ be a highly computable graph. Given $v$ and $r$ as before,
	we can compute the finite graph $G(v,r)$.
\end{fact}

\end{comment}
{} Now we prove some results related to the removal of finite sets of
edges from a graph. 
\begin{lemma}
\label{lem:calculabilidad-de-vecinos}There is an algorithm which,
on input the description of a moderately computable graph $\G$, and
a finite set of edges $E$ in $\G$, outputs a list of all vertices
in $G-E$ which are incident to some edge in $E$. 
\end{lemma}

\begin{proof}
On input $G$ and $E$, we first compute the finite set $V$ of all
vertices in $V(G)$ incident to some edge in $E$. Then, for each
$v\in V$, we can decide whether $v$ lies in $G-E$ by using the
computability of the degree function:
\begin{enumerate}
	\item If $\deg_{G}(v)=\infty$, then $v$ must have some incident edge not
	in the finite set $E$. It follows that $v$ lies in $G-E$.
	\item If $\deg_{G}(v)$ is a finite number, then we use this information
	and the decidability of the incidence relation to compute the set
	of all edges incident to $v$. After doing this, we can check whether
	$v$ has some incident edge not in $E$. Observe that this is equivalent
	to ask whether $v$ lies in $G-E$. 
\end{enumerate}
After checking each element $v\in V$, the algorithm outputs a list
of those in $G-E$.
\end{proof}
\begin{lemma}
\label{lem:semidecidable-finite-connected-component}There is an algorithm
which, on input the description of a connected moderately computable graph $\G$,
and a finite set of edges $E$ in $\G$, halts if and only if $\G-E$
has some finite connected component. 
\end{lemma}

\begin{proof}
On input $G$ and $E$, the algorithm proceeds as follows. First,
compute a vertex $v$ which is incident to some edge in $E$. Then,
for each pair of natural numbers $r$ and $s$, compute the graph
$G(v,r,s)$ and check whether the finite graph $G(v,r,s)-E$ has a
connected component $G'$ verifying the following two conditions:
\begin{enumerate}
	\item Every vertex in $G'$ has finite degree in $G$.
	\item Every vertex in $G'$ has all its incident edges from $G$ in $G(v,r,s)$.
\end{enumerate}
Given a connected component of $G(v,r,s)-E$, we can decide whether
these conditions are satisfied. Indeed, for the first condition we
just need to compute $\deg_{G}(v)$ for each $v\in V(G')$ and check
that it is finite, while for the second one we just need to compare
$\deg_{G}(v)$ and $\deg_{G'}(v)$ for each $v\in V(G')$. If such
a connected component exists, then the algorithm halts and concludes
that $G-E$ has a finite connected component (which indeed is $G'$).

Let us check that this algorithm is correct. If the algorithm halts
and such a component $G'$ exists for some $r_{0}$ and $s_{0}$,
then $G'$ will continue to be a finite connected component of $G(v,r,s)-E$
for all $r\geq r_{0}$, $s\geq s_{0}$ because no new edges incident
to vertices in $G'$ can appear, and thus $G'$ is a finite connected
component of $G-E$. On the other hand it is clear that if $G-E$
has some connected component $G'$, then it has the form specified
above for some $r_{0}$ and $s_{0}$ big enough. 
\end{proof}
It is known that on a highly computable graph with one end, it is
algorithmically decidable whether the removal of a finite set of edges
leaves a connected graph \cite[Lemma 4.4]{carrasco-vargas_geometric_2023}.
Here we extend this result to moderately computable graphs.
\begin{lemma}
\label{lem:decidable-connected-complement-one-end}There is an algorithm
which, on input the description of a connected moderately computable graph $\G$
which has one end, and a finite set of edges $E$ in $\G$, decides
whether $\G-E$ is connected.
\end{lemma}

\begin{proof}
By \Cref{lem:semidecidable-finite-connected-component}, it suffices
to exhibit an algorithm which on input a description of $\G$ and
$E$ as in the statement, halts if and only if $\G-E$ is connected.
On input $G$ and $E$, the algorithm proceeds as follows. First,
execute the procedure in \Cref{lem:calculabilidad-de-vecinos}
to compute the finite set $\{v_{1},\dots,v_{n}\}$ of all vertices
incident to some edge in $E$, and which also lie in $\G-E$. Define
$\G_{i}$ as the connected component in $\G-E$ containing $v_{i}$,
$i\in\{1,\dots,n\}$. 

Observe that the graph $\G-E$ is connected if and only if all the
components $\G_{i}$ are equal. Moreover, a component $G_{i}$ equals
$G_{j}$ if and only if there is a path in $G$ whose initial vertex
is $v_{i}$, final vertex $v_{j}$, and which does not visit edges
among $E$. Thus we just need to search for paths exhaustively, and
halt the algorithm when, for every pair of $i,j\in\{1,\dots,n\}$,
$i\ne j$, we have found a path whose initial vertex is $v_{i}$,
whose initial vertex is $v_{j}$, and which does not visit edges in
$E$. 
\end{proof}
In what follows, we will apply these results to the properties that
we defined for paths. More specifically, we will prove that it is
algorithmically decidable whether a finite path is extensible to a
one-way or two-way infinite Eulerian path. 

\subsection{Proof of \Cref{thm:right=000020bi=000020extensible=000020is=000020decidable}
}\label{subsec:Decidability-of-path-properties}

In this subsection we prove \Cref{thm:right=000020bi=000020extensible=000020is=000020decidable},
which asserts that for each moderately computable graph $G$ satisfying
$\os$ (resp. $\e$), it is algorithmically decidable whether a finite
path can be extended to a one-way (resp. two-way) infinite Eulerian
path on $G$. We start with the proof of \Cref{thm:right=000020bi=000020extensible=000020is=000020decidable}
in the case of paths extensible to one-way infinite Eulerian paths:
\begin{proposition}
\label{prop:decidability-right-extensible-moderately-computable}There
is an algorithm which, on input the description of a moderately computable
graph $\G$ satisfying $\os$, the information of whether $\G$ has
some vertex with odd degree, and a finite path, decides whether the
path can be extended to a one-way infinite path on $G$. 
\end{proposition}

\begin{proof}
Let $G$ be a graph as in the statement, and let $t$ be a path as
in the statement. By \Cref{thm:e1-characterization-of-trails}
it suffices to check whether the following three conditions are satisfied:
\begin{enumerate}
	\item $\G-t$ is connected. 
	\item The initial vertex of $t$ is distinguished in $\G$.
	\item There is an edge $e$ incident to the final vertex of $t$ which was
	not visited by $t$. 
\end{enumerate}
Condition (1) can be checked with the algorithm in \Cref{lem:decidable-connected-complement-one-end}.
For the condition (2), let us recall that if $G$ has a vertex with
odd degree, then this is the only distinguished vertex of $G$, and
if $G$ has no vertex with odd degree, then a vertex is distinguished
if and only if it has infinite degree (\Cref{def:distinguished}).
It follows the information of whether $G$ has a vertex with odd degree
plus the computability of the vertex degree function are sufficient
to decide whether the initial vertex of $t$ is distinguished. 

Finally, observe that $t$ satisfies (3) exactly when its final vertex
lies in $G-t$. This can be checked with the algorithm in \Cref{lem:calculabilidad-de-vecinos}. 
\end{proof}
\begin{remark}
\label{rem:uniformity} The previous result shows that for each individual
graph $G$ satisfying $\os$, the problem of deciding whether a path
can be extended to a one-way infinite Eulerian path is decidable.
However, this algorithm is not uniform on the graph, in the sense
that we must tell the algorithm whether $G$ has a vertex of odd degree
or not. This can not be improved:
\end{remark}
\begin{proposition}
There is no algorithm which, on input the description of a moderately
computable graph satisfying $\os$, decides whether it has a vertex
with odd degree or not.
\end{proposition}
The proof is left to the interested reader. We observe, however, that
the algorithm is uniform when restricted to highly computable graphs.
This is because a locally finite graph satisfying $\os$ must always
have a vertex with odd degree.

\begin{proposition}
\label{prop:decidability-right-extensible}There is an algorithm which,
on input the description of a highly computable graph $\G$ satisfying
$\os$, and a finite path, decides whether it is right-extensible. 
\end{proposition}

\begin{proof}
Let us recall that a locally finite graph satisfying $\os$ has exactly
one vertex with odd degree (\Cref{def:egw-conditions-e1}). On
input the graph $G$ and a path $t$, we just have to use the algorithm
in \Cref{prop:decidability-right-extensible-moderately-computable},
and inform the algorithm that the input graph $G$ has a vertex with
odd degree. 
\end{proof}
Now we move on to consider graphs satisfying $\e$, and two-way infinite
Eulerian paths. The following result proves \Cref{thm:right=000020bi=000020extensible=000020is=000020decidable}
in the case of paths extensible to two-way infinite Eulerian paths.
Observe that in this case the algorithm is uniform for moderately
computable graphs.
\begin{proposition}
\label{prop:decidability-bi-extensible} There is an algorithm which,
on input the description of a moderately computable graph $\G$ satisfying
$\e$, and a finite path, decides whether it can be extended to a two-way
infinite Eulerian path. 
\end{proposition}

\begin{proof}
Let $G$ be a graph as in the statement, and let $t$ be a path as
in the statement. By \Cref{thm:e2-characterization-of-trails}
it suffices to verify whether the following three conditions are satisfied: 
\begin{enumerate}
	\item \label{enu:condicion=000020uno}$\G-t$ has no finite connected components.
	\item There is an edge $e$ incident to the final vertex of $t$ which was
	not visited by $t$. 
	\item There is an edge $f\ne e$ incident to the initial vertex of $t$
	which was not visited by $t$. 
\end{enumerate}
Conditions (2) and (3) can be checked by computing the vertex degree
of the initial and final vertex of $t$, and then going into case
by case considerations. The details are left to the reader. 

Once we checked the conditions (2) and (3), we proceed to check \eqref{enu:condicion=000020uno}.
By \Cref{prop:bi-extensible-alternative-characterization}, this
condition is equivalent to the following: 
\begin{enumerate}[resume]
	\item \label{enu:condicion=000020uno=000020alternativa}Every connected
	component of $G-t$ contains either the initial or the final vertex
	of $t$.
\end{enumerate}
In order to check whether these these two equivalent conditions are
satisfied by $t$, we run two procedures simultaneously. The first
halts when $t$ fails to satisfy \eqref{enu:condicion=000020uno},
while the second halts when $t$ satisfies \eqref{enu:condicion=000020uno=000020alternativa}.

Let $E$ be the set of edges visited by $t$. The first procedure
is the algorithm in \Cref{lem:semidecidable-finite-connected-component}
with input $G$ and $E$. This algorithm halts if and only if $G-t$
contains some finite connected component, that is, if $t$ fails to
satisfy \eqref{enu:condicion=000020uno}. 

We now describe the second procedure, which halts if and only if \eqref{enu:condicion=000020uno=000020alternativa}
holds. The idea is that this condition can be detected by searching
for paths exhaustively. More precisely, we start by using the algorithm
in \Cref{lem:calculabilidad-de-vecinos} to compute the finite
set $\{v_{1},\dots,v_{n}\}$ of all vertices incident to some edge
in $E$ which also lie in $\G-t$. Now let $\G_{i}$ be the connected
component in $\G-t$ containing $v_{i}$. We label these vertices
so that $v_{1}$ is the initial vertex of $t$, and $v_{n}$ is the
final vertex of $t$. Then observe that \eqref{enu:condicion=000020uno=000020alternativa}
is satisfied if and only if for every $i\in\{2,\dots,n-1\}$, there
is a path which does not visit edges from $t$, and which joins $v_{i}$
to either $v_{1}$ or $v_{n}$. Thus it suffices to enumerate paths
in $G$ exhaustively, and stop the algorithm once the condition mentioned
before is satisfied. 
\end{proof}

\subsection{Proof of \Cref{thm:effective=000020egw}}

In this subsection we prove \Cref{thm:effective=000020egw}.
This is a straightforward application of \Cref{thm:right=000020bi=000020extensible=000020is=000020decidable},
proved in the previous subsection. 
\begin{proposition}[\Cref{thm:effective=000020egw}]
Let $G$ be a moderately computable graph satisfying $\os$ (resp.
$\e$). Then $G$ admits a computable one-way (resp. two-way) infinite
Eulerian path. 
\end{proposition}

\begin{proof}
We prove the result for one-way infinite Eulerian paths, the proof
for two-way infinite Eulerian paths is very similar. Let $G$ be a
moderately computable graph satisfying $\os$. Without loss of generality,
we can assume that the edge set of $G$ is indexed as $E(G)=\{e_{n}\mid n\in\N\}$
(it is a standard fact that any decidable and infinite subset of $\N$
admits a computable bijection onto $\N$). 

We recall that in \Cref{prop:extensibility-one-sided}, we showed
how to construct a one-way infinite Eulerian path on $G$. More precisely,
we defined a sequence of paths $(t_{n})_{n\in\N}$ satisfying the
following conditions for each $n\in\N$:
\begin{enumerate}
	\item The path $t_{n}$ is right-extensible (see \Cref{def:right-extensible}
	and \Cref{thm:e1-characterization-of-trails}).
	\item The path $t_{n}$ visits $e_{n}$.
	\item The path $t_{n+1}$ extends $t_{n}$. 
\end{enumerate}
Such a sequence of paths defines a one-way infinite Eulerian path.
In order to show the existence of a \emph{computable} one-way infinite
Eulerian path, it suffices to show that such a sequence of paths can
be computed. We already proved that the property of being right-extensible
is decidable (\Cref{thm:right=000020bi=000020extensible=000020is=000020decidable}
and \Cref{thm:e1-characterization-of-trails}). It follows from
this observation that such a sequence of paths can be computed recursively,
by performing an exhaustive search in each step. 

We describe the argument more precisely. A path $t_{0}$ as required
can be computed by exhaustive search: the required conditions are
decidable (it is right-extensible, and it visits $e_{0}$), its existence
is guaranteed, and thus the search will eventually stop. Now, assuming
that $t_{0},\dots,t_{n}$ have been already defined, we compute $t_{n+1}$
in the same manner. The required conditions are decidable (it is right-extensible,
it extends $t_{n}$, and it visits $e_{n+1}$), and its existence
is already guaranteed, so the search will eventually stop.
\end{proof}
\begin{remark}
The uniformity considerations associated to \Cref{thm:effective=000020egw}
are the same as those of \Cref{thm:right=000020bi=000020extensible=000020is=000020decidable}.
That is, in the case of two-way infinite Eulerian paths, the algorithm
is uniform for moderately computable graphs. In the case of one-way
infinite Eulerian paths, the algorithm is uniform only for highly
computable graphs. 
\end{remark}

\section{Some remarks about moderately computable graphs}\label{sec:remarks}

In this section we make some comments about the notion of moderately
computable graph. In particular, we emphasize some differences that arise in comparison to highly computable graphs. Let us first review an elementary example where
our results apply:
\begin{example}
Let $G_{0}$ be the graph consisting of a single vertex, and an infinite
and countable number of loops joining this vertex to itself. Note
that the graph $G_{0}$ admits both a one-way infinite Eulerian path,
and a two-way infinite Eulerian path. Moreover, this graph is moderately
computable. 

\begin{figure}[H]
	\begin{center}
		\begin{tikzpicture}
			\node (3) [circle,draw] {}  ;
			\path  (3)   edge[distance=45mm] node[above]  {$$} (3);
			\path  (3)   edge[distance=40mm] node[above]  {$$} (3);
			\path  (3)   edge[distance=35mm] node[above]  {$$} (3);
			\path  (3)   edge[distance=30mm] node[above]  {$$} (3);
			\path  (3)   edge[distance=25mm] node[below]  {$\vdots$} (3);
		\end{tikzpicture}
	\end{center}
	
	\caption{A representation of the graph $G_{0}$. }\label{fig:dibujito}
\end{figure}
\end{example}

This example illustrates that the Erdős, Grünwald, and Weiszfeld theorem
naturally applies to graphs with loops and multiple edges between
vertices. For this reason, we have given a definition of moderately
computable graph which ensures that the generality of the result is
not lost. However, allowing loops has certain influence
in the computability of certain operations. In particular, adding loops allows us to give a simple proof of the following result:
\begin{proposition}
\label{prop:moderately=000020computable=000020graph=000020with=000020uncomputable=000020distance}There
is a (non simple) moderately computable graph for which the distance
function is not computable. 
\end{proposition}

\begin{proof}
It is well known that for computable, simple, and locally finite graphs,
the distance function is not guaranteed to be computable (see \cite{calvert_Distance_2011}).
Let $G$ be a computable, simple, and locally finite graph for which
the distance function is not computable. Now let $G'$ be the graph
obtained by joining to each vertex in $G$ a countable and infinite
number of loops. The new graph is still computable. It is also moderately
computable, as the vertex degree function has constant value $\infty$.
However, the distance function on the new graph $G'$ is the same
as the distance function on the graph $G$, and thus it is not computable.
\end{proof}
It is well known that in a highly computable graph, the distance function
is computable. The previous result shows that this property is lost
in the larger class of moderately computable graphs. However, our
example was obtained by adding loops, and we do not know whether the
situation changes if we require the graph to be simple:
\begin{question}
Is there a simple moderately computable graph, with an uncomputable
distance function?
\end{question}

We now make some comments about the the computability of the vertex degree function in comparison to the computability of the number of neighbors of vertices. We denote by $N_{G}(v)$
the set of neighbors of a vertex $v$ in the graph $G$. That is,
$N_{G}(v)$ is the set of all $u\in V(G)$ such that there is an edge
joining $\ensuremath{u}$ to $\ensuremath{v}$. For simple graphs,
we have the equality $\deg_{G}(v)=|N_{G}(v)|$ and thus requiring
the vertex function to be computable is the same as requiring the
function $V(G)\to\N,v\mapsto|N_{G}(v)|$ to be computable. For locally
finite graphs, with loops and multiple edges allowed, it is easy to
see that the computability of the vertex degree function implies the
computability of the function $V(G)\to\N,v\mapsto|N_{G}(v)|$. However, this
property does not hold anymore beyond locally finite graphs:
\begin{proposition}
There is a moderately computable graph where each vertex has finitely
many neighbors, but the function $V(G)\to\N,v\mapsto|N_{G}(v)|$ is
not computable.
\end{proposition}

\begin{proof}
Let $G$ and $G'$ be the graphs from \Cref{prop:moderately=000020computable=000020graph=000020with=000020uncomputable=000020distance}.
We claim that $G'$ has the property in the statement. Indeed, the
graphs $G$ and $G'$ have the same vertex set, and for each vertex
$v$ we have $|N_{G'}(v)|=|N_{G}(v)|+1$ (the $+1$ comes from the
loops). This shows that if $v\mapsto|N_{G'}(v)|$ was a computable
function, then $G$ would be a highly computable graph. But this would
imply that its distance function is computable, a contradiction to
the hypothesis on $G$.
\end{proof}
\bibliographystyle{abbrv}
\bibliography{bibliography}
\end{document}